\documentclass[12pt,a4paper]{article}%
\usepackage{amssymb}
\usepackage{amsfonts}
\usepackage{amsmath}
\usepackage{graphicx}
\usepackage[singlespacing]{setspace}
\usepackage{geometry}
\usepackage{url}
\usepackage{nopageno}%
\providecommand{\U}[1]{\protect\rule{.1in}{.1in}}
\newtheorem{theorem}{Theorem}

\newtheorem{proposition}[theorem]{Proposition}

\newenvironment{proof}[1][Proof]{\noindent\textbf{#1.} }{\ \rule{0.5em}{0.5em}}
\begin{document}

\title{\textbf{The Welfare of Ramsey Optimal Policy Facing Auto-Regressive
Shocks\bigskip}\\Post Print: Economics Bulletin (2020), 40(2), pp. 1797-1803.}
\author{Jean-Bernard Chatelain\thanks{Paris School of Economics, Universit\'{e} Paris
I Pantheon Sorbonne, PjSE, 48 Boulevard Jourdan, 75014 Paris. Email:
jean-bernard.chatelain@univ-paris1.fr} and Kirsten Ralf\thanks{ESCE
International Business School, INSEEC\ U. Research Center, 10 rue Sextius
Michel, 75015 Paris, Email: Kirsten.Ralf@esce.fr.}}
\date{June 24, 2020}
\maketitle

\begin{abstract}
With non-controllable auto-regressive shocks, the welfare of Ramsey optimal
policy is the solution of a single Riccati equation of a linear quadratic
regulator. The existing theory by Hansen and Sargent (2007) refers to an
additional Sylvester equation but miss another equation for computing the
block matrix weighting the square of non-controllable variables in the welfare
function. There is no need to simulate impulse response functions over a long
period, to compute period loss functions and to sum their discounted value
over this long period, as currently done so far. Welfare is computed for the
case of the new-Keynesian Phillips curve with an auto-regressive cost-push shock.

\textbf{JEL\ classification numbers}: C61, C62, C73, E47, E52, E61, E63.

\newpage

\end{abstract}

\section{Introduction}

Dynamic stochastic general equilibrium (DSGE) models include auto-regressive
shocks (Smets and Wouters (2007)). For computing the welfare of Ramsey optimal
policy in DSGE models, one simulates impulse response functions over a long
period, one computes period loss functions and one sums their discounted value
over this long period.

Since Anderson \textit{et al.} (1996) and Hansen and Sargent (2007), the
available theory uses a Riccati equation for controllable variables and a
Sylvester equation for non-controllable variables in order to find the optimal
policy rule and the optimal initial condition for non-predetermined variables.
However, the matrix of the value function allowing to compute welfare is
incomplete. Even worse, computing welfare loss using only the two matrices
solutions of Riccati equation and of Sylvester equation may lead to a strictly
positive value, which is impossible. A third equation is missing in order to
find the matrix related to the squares of the non-controllable variable in the
value function.

We include in the Lagrangian the Lagrange multiplier times the dynamic
equation of the non-controllable variables. This Lagrange multiplier is
omitted in Anderson \textit{et al.} (1996), p.202. Once this Lagrangian
multiplier is included, the symmetry of the Hamiltonian matrix for the full
system of controllable and non-controllable variables is restored. The value
function is the solution of a Riccati equation for matrices related to
controllable \emph{and} non-controllable variables.

In Anderson \textit{et al.} (1996), the Riccati equation only on controllable
variables and the Sylvester equation only on non-controllable variables
corresponds to two block matrix of the solution of our Riccati equation. The
missing block matrix for computing welfare related to the square of
non-controllable variables is found solving this Riccati equation.

This numerical solution of this Riccati equation is coded in the linear
quadratic regulator instruction \texttt{lqr} in SCILAB. We compute the welfare
of Ramsey optimal policy for the new-Keynesian Phillips curve with an
auto-regressive cost-push shock (Gali (2015)).

\section{The Welfare of Ramsey optimal policy}

To derive Ramsey optimal policy a Stackelberg leader-follower model is
analyzed where the government is the leader and the private sector is the
follower. Let $\mathbf{k}_{t}$ be an $n_{k}\times1$ vector of controllable
predetermined state variables with initial conditions $\mathbf{k}_{0}$ given,
$\mathbf{x}_{t}$ an $n_{x}\times1$ vector of non-predetermined endogenous
variables free to jump at $t$ without a given initial condition for
$\mathbf{x}_{0}$, put together in the $\left(  n_{k}+n_{x}\right)  \times
1$vector $\mathbf{y}_{t}=(\mathbf{k}_{t}^{T},\mathbf{x}_{t}^{T})^{T}$. The
$n_{u}\times1$ vector $\mathbf{u}_{t}$ denotes government policy instruments.
We include an $n_{z}\times1$ vector of non-controllable autoregressive shocks
$\mathbf{z}_{t}$. All variables are expressed as absolute or proportional
deviations from a steady state.

The policy maker maximizes the following quadratic function (minimizes the
quadratic loss) subject to an initial condition for $\mathbf{k}_{0}$ and
$\mathbf{z}_{0}$, but not for $\mathbf{x}_{0}$:%

\begin{equation}
-\frac{1}{2}%
{\displaystyle\sum\limits_{t=0}^{+\infty}}
\beta^{t}\left(  \mathbf{y}_{t}^{T}\mathbf{Q}_{yy}\mathbf{y}_{t}%
+2\mathbf{y}_{t}^{T}\mathbf{Q}_{yz}\mathbf{z}_{t}+\mathbf{z}_{t}^{T}%
\mathbf{Q}_{zz}\mathbf{z}_{t}+\mathbf{u}_{t}^{T}\mathbf{R}_{uu}\mathbf{u}%
_{t}\right)  \text{ }%
\end{equation}
where $\beta$ is the policy maker's discount factor. The policymaker's
preferences are the relative weights included in the matrices $\mathbf{Q}$ and
$\mathbf{R}$. $\mathbf{Q}_{yy}\geq\mathbf{0}$ is a $\left(  n_{k}%
+n_{x}\right)  \times\left(  n_{k}+n_{x}\right)  $ positive symmetric
semi-definite matrix, $\mathbf{R}_{uu}>\mathbf{0}$ is a $p\times p$
\emph{strictly} positive symmetric definite matrix, so that the policy maker
has at least a very small concern for the volatility of policy instruments.
The policy transmission mechanism of the private sector's behavior is
summarized by this system of equations:%

\begin{equation}
\left(
\begin{array}
[c]{c}%
E_{t}\mathbf{y}_{t+1}\\
\mathbf{z}_{t+1}%
\end{array}
\right)  =\left(
\begin{array}
[c]{cc}%
\mathbf{A}_{yy} & \mathbf{A}_{yz}\\
\mathbf{0} & \mathbf{A}_{zz}%
\end{array}
\right)  \left(
\begin{array}
[c]{c}%
\mathbf{y}_{t}\\
\mathbf{z}_{t}%
\end{array}
\right)  +\left(
\begin{array}
[c]{c}%
\mathbf{B}_{yu}\\
\mathbf{0}%
\end{array}
\right)  \mathbf{u}_{t}, \label{Transmission}%
\end{equation}
where $\mathbf{A}$ is an $\left(  n_{k}+n_{x}+n_{z}\right)  \times\left(
n_{k}+n_{x}+n_{z}\right)  $ matrix and $\mathbf{B}$ is the $\left(
n_{k}+n_{x}+n_{z}\right)  \times p$ matrix of marginal effects of policy
instruments $\mathbf{u}_{t}$ on next period policy targets $\mathbf{y}_{t+1}$.
The certainty equivalence principle of the linear quadratic regulator allows
us to work with a non-stochastic model (Anderson \textit{et al.} (1996)).
Anderson \textit{et al.} (1996) is word by word Hansen and Sargent (2007)
chapter 5, so we refer only to Anderson \textit{et al.} (1996) in what follows.

The government chooses sequences $\left\{  \mathbf{u}_{t},\mathbf{x}%
_{t},\mathbf{k}_{t+1}\right\}  _{t=0}^{+\infty}$ taking into account the
policy transmission mechanism (\ref{Transmission}) and boundary conditions
detailed below.

Essential boundary conditions are the initial conditions of predetermined
variables $\mathbf{k}_{0}$ and $\mathbf{z}_{0}$ which are given. Natural
boundary conditions are chosen by the policy maker to anchor the unique
optimal initial values of the private sector's forward-looking variables. The
policy maker's Lagrange multipliers of the private sector's forward (Lagrange
multipliers) variables are \emph{predetermined at the value zero: }%
$\frac{\partial L}{\partial\mathbf{x}_{0}}=\mathbf{\mu}_{\mathbf{x},t=0}=0$ in
order to determine the unique optimal initial value $\mathbf{x}_{0}%
=\mathbf{x}_{0}^{\ast}$ of the private sector's forward variables.

Anderson \textit{et al.} (1996) assume a bounded discounted quadratic loss function:%

\begin{equation}
E\left(
{\displaystyle\sum\limits_{t=0}^{+\infty}}
\beta^{t}\left(  \mathbf{y}_{t}^{T}\mathbf{y}_{t}+\mathbf{z}_{t}^{T}%
\mathbf{z}_{t}+\mathbf{u}_{t}^{T}\mathbf{u}_{t}\right)  \right)
<+\infty\text{ }%
\end{equation}
which implies%

\begin{align*}
\underset{t\rightarrow+\infty}{\lim}\beta^{t}\mathbf{z}_{t}  &  =\mathbf{z}%
^{\ast}=\mathbf{0}\text{, }\mathbf{z}_{t}\text{ bounded,}\\
\underset{t\rightarrow+\infty}{\lim}\beta^{t}\mathbf{y}_{t}  &  =\mathbf{y}%
^{\ast}=\mathbf{0}\Leftrightarrow\underset{t\rightarrow+\infty}{\lim}%
\frac{\partial L}{\partial\mathbf{y}_{t}}=\mathbf{0}=\underset{t\rightarrow
+\infty}{\lim}\beta^{t}\mathbf{\mu}_{t}\text{, }\mathbf{\mu}_{t}\text{
bounded.}%
\end{align*}

The bounded discounted quadratic loss function implies to select eigenvalues
of the dynamic system such that $\left\vert \left(  \beta\lambda_{i}%
^{2}\right)  ^{t}\right\vert <\left\vert \beta\lambda_{i}^{2}\right\vert <1$
or equivalently such that: $\left\vert \lambda_{i}\right\vert <1/\sqrt{\beta}%
$. A preliminary step is to multiply matrices by $\sqrt{\beta}$ as follows:
$\sqrt{\beta}\mathbf{A}_{yy}$\ $\sqrt{\beta}\mathbf{B}_{y}$ in order to apply
the formulas of Riccati equations for the \emph{non}-discounted linear
quadratic regulator augmented by auto-regressive shocks.

\textbf{Assumption 1:} The matrix pair ($\sqrt{\beta}\mathbf{A}_{yy}$%
\ $\sqrt{\beta}\mathbf{B}_{yu}$) is Kalman controllable if the controllability
matrix has full rank:
\begin{equation}
\text{rank }\left(  \sqrt{\beta}\mathbf{B}_{yu}\text{ \ }\beta\mathbf{A}%
_{yy}\mathbf{B}_{yu}\text{ \ }\beta^{\frac{3}{2}}\mathbf{A}_{yy}^{2}%
\mathbf{B}_{yu}\text{ \ ... \ }\beta^{\frac{n_{k}+n_{x}}{2}}\mathbf{A}%
_{yy}^{n_{k}+n_{x}-1}\mathbf{B}_{yu}\right)  =n_{k}+n_{x}.
\end{equation}

\textbf{Assumption 2:} The system is can be stabilized when the transition
matrix $\mathbf{A}_{zz}$ for the non-controllable auto-regressive variables
has eigenvalues such that $\left\vert \lambda_{i}\right\vert <1/\sqrt{\beta}$.

The policy maker's choice can be solved with Lagrange multipliers. The
Lagrangian includes not only the constraints of the private sector's policy
transmission mechanisms multiplied by their respective Lagrange multipliers
$2\beta^{t+1}\mathbf{\mu}_{t+1}$, \textbf{BUT\ ALSO} \ \emph{the constraints
of the non-controllable variables dynamics with their respective Lagrange
multiplier }$2\beta^{t+1}\mathbf{\nu}_{t+1}$\emph{, which were omitted in
Anderson et al. (1996), p.202}.%

\begin{equation}
-\frac{1}{2}%
{\displaystyle\sum\limits_{t=0}^{+\infty}}
\begin{array}
[c]{c}%
\beta^{t}\left[  \mathbf{y}_{t}^{T}\mathbf{Q}_{yy}\mathbf{y}_{t}%
+2\mathbf{y}_{t}^{T}\mathbf{Q}_{yz}\mathbf{z}_{t}+\mathbf{z}_{t}^{T}%
\mathbf{Q}_{zz}\mathbf{z}_{t}+\mathbf{u}_{t}^{T}\mathbf{R}_{uu}\mathbf{u}%
_{t}\right]  +\\
2\beta^{t+1}\mathbf{\mu}_{t+1}\left[  \mathbf{A}_{yy}\mathbf{y}_{t}%
+\mathbf{A}_{yz}\mathbf{z}_{t}+\mathbf{B}_{yu}u_{t}-\mathbf{y}_{t+1}\right]
+\\
2\beta^{t+1}\mathbf{\nu}_{t+1}\left[  \mathbf{A}_{zz}\mathbf{z}_{t}%
+\mathbf{0}_{z}u_{t}-\mathbf{z}_{t+1}\right]  .
\end{array}
\text{ }%
\end{equation}

The first order conditions are:%

\begin{align*}
\frac{\partial L}{\partial x_{t}}  &  =\mathbf{Rx}_{t}+\beta\mathbf{B\gamma
}_{t+1}=0\Rightarrow\mathbf{x}_{t}=-\beta\mathbf{R}^{-1}\mathbf{B\gamma}%
_{t+1}\\
\frac{\partial L}{\partial\pi_{t}}  &  =\mathbf{Q\pi}_{t}+\beta\mathbf{A\gamma
}_{t+1}-\mathbf{\gamma}_{t}=0\\
\frac{\partial L}{\partial z_{t}}  &  =\beta\mathbf{\gamma}_{t+1}%
\mathbf{A}_{yz}+\beta\mathbf{\delta}_{t+1}\mathbf{A}_{zz}-\mathbf{\delta}%
_{t}=0
\end{align*}

The policy instrument are substituted by $\mathbf{x}_{t}=-\beta\mathbf{R}%
^{-1}\mathbf{B\gamma}_{t+1}$ in the transmission mechanism equation. The
Hamiltonian of the linear quadratic regulator has the usual block matrices on
left hand side and right hand side:%

\[
L=\left(
\begin{array}
[c]{cc}%
\mathbf{I} & \mathbf{-}\beta\mathbf{B}_{(y,z)u}\mathbf{R}_{uu}^{-1}%
\mathbf{B}_{(y,z)u}^{T}\\
\mathbf{0} & \beta\mathbf{A}^{T}%
\end{array}
\right)  \text{ and }N=\left(
\begin{array}
[c]{cc}%
\mathbf{A} & \mathbf{0}\\
\mathbf{-Q} & \mathbf{I}%
\end{array}
\right)
\]

with this particular block decomposition between controllable variables
$\mathbf{y}_{t}$ and non-controllable variables $\mathbf{z}_{t}$:%

\[
\left(
\begin{array}
[c]{cccc}%
\mathbf{I} & \mathbf{0} & \mathbf{-}\beta\mathbf{B}_{yu}\mathbf{R}_{uu}%
^{-1}\mathbf{B}_{yu}^{T} & \mathbf{0}\\
\mathbf{0} & \mathbf{I} & \mathbf{0} & \mathbf{0}\\
\mathbf{0} & \mathbf{0} & \beta\mathbf{A}_{yy} & \mathbf{0}\\
\mathbf{0} & \mathbf{0} & \beta\mathbf{A}_{yz} & \beta\mathbf{A}_{zz}%
\end{array}
\right)  \left(
\begin{array}
[c]{c}%
y_{t+1}\\
z_{t+1}\\
\mu_{t+1}\\
\nu_{t+1}%
\end{array}
\right)  =\left(
\begin{array}
[c]{cccc}%
\mathbf{A}_{yy} & \mathbf{A}_{yz} & \mathbf{0} & \mathbf{0}\\
\mathbf{0} & \mathbf{A}_{zz} & \mathbf{0} & \mathbf{0}\\
-\mathbf{Q}_{yy} & -\mathbf{Q}_{yz} & \mathbf{I} & \mathbf{0}\\
-\mathbf{Q}_{yz} & -\mathbf{Q}_{zz} & \mathbf{0} & \mathbf{I}%
\end{array}
\right)  \left(
\begin{array}
[c]{c}%
y_{t}\\
z_{t}\\
\mu_{t+1}\\
\nu_{t+1}%
\end{array}
\right)
\]

The specificity of non-controllable variables is that the following matrix
includes three blocks with zeros, which is not the case for controllable variables:%

\[
-\beta\left(
\begin{array}
[c]{c}%
\mathbf{B}_{yu}\\
\mathbf{0}%
\end{array}
\right)  \left(  \mathbf{R}_{uu}^{-1}\right)  \left(
\begin{array}
[c]{c}%
\mathbf{B}_{yu}\\
\mathbf{0}%
\end{array}
\right)  ^{T}=\left(
\begin{array}
[c]{cc}%
\mathbf{-}\beta\mathbf{B}_{yu}\mathbf{R}_{uu}^{-1}\mathbf{B}_{yu}^{T} &
\mathbf{0}\\
\mathbf{0} & \mathbf{0}%
\end{array}
\right)
\]

If $\mathbf{L}$ is non-singular, the Hamiltonian matrix $\mathbf{H=L}%
^{-1}\mathbf{N}$ is a symplectic matrix. With the equations of the Lagrange
multipliers $\mathbf{\nu}_{t+1}$, all the roots $\rho_{i}$ of $A_{zz}$ have
their mirror roots $\left(  1/\beta\rho_{i}\right)  $ which were all missing
in Anderson \emph{et al.} (1996).

The value function for welfare involve the matrix $\mathbf{P}$ such that:
\[
L_{t=0}=\left(
\begin{array}
[c]{c}%
y_{0}\\
z_{0}%
\end{array}
\right)  ^{T}\left(
\begin{array}
[c]{cc}%
\mathbf{P}_{yy} & \mathbf{P}_{yz}\\
\mathbf{P}_{yz} & \mathbf{P}_{zz}%
\end{array}
\right)  \allowbreak\left(
\begin{array}
[c]{c}%
y_{0}\\
z_{0}%
\end{array}
\right)
\]

A stabilizing solution of the Hamiltonian system satisfies (Anderson \emph{et
al.} (1996)):%
\begin{equation}
\frac{\partial L}{\partial\mathbf{y}_{t=0}}=\mathbf{\mu}_{0}=\mathbf{P}%
_{y}\mathbf{y}_{0}+\mathbf{P}_{z}\mathbf{z}_{0}.
\end{equation}

The optimal rule of the augmented linear quadratic regulator is:%
\begin{equation}
\mathbf{u}_{t}=\mathbf{F}_{y}\mathbf{y}_{t}+\mathbf{F}_{z}\mathbf{z}_{t}.
\end{equation}

The matrix $\mathbf{P}$ is solution of this Riccati equation:%
\begin{align*}
\left(
\begin{array}
[c]{cc}%
\mathbf{P}_{yy} & \mathbf{P}_{yz}\\
\mathbf{P}_{yz} & \mathbf{P}_{zz}%
\end{array}
\right)  \allowbreak &  =\left(
\begin{array}
[c]{cc}%
\mathbf{Q}_{yy} & \mathbf{Q}_{yz}\\
\mathbf{Q}_{yz} & \mathbf{Q}_{zz}%
\end{array}
\right)  +\beta\left(
\begin{array}
[c]{cc}%
\mathbf{A}_{yy} & \mathbf{A}_{yz}\\
\mathbf{0} & \mathbf{A}_{zz}%
\end{array}
\right)  ^{T}\left(
\begin{array}
[c]{cc}%
\mathbf{P}_{yy} & \mathbf{P}_{yz}\\
\mathbf{P}_{yz} & \mathbf{P}_{zz}%
\end{array}
\right)  \allowbreak\left(
\begin{array}
[c]{cc}%
\mathbf{A}_{yy} & \mathbf{A}_{yz}\\
\mathbf{0} & \mathbf{A}_{zz}%
\end{array}
\right) \\
&  -\allowbreak\beta\left(
\begin{array}
[c]{cc}%
\mathbf{A}_{yy} & \mathbf{A}_{yz}\\
\mathbf{0} & \mathbf{A}_{zz}%
\end{array}
\right)  ^{T}\left(
\begin{array}
[c]{cc}%
\mathbf{P}_{yy} & \mathbf{P}_{yz}\\
\mathbf{P}_{yz} & \mathbf{P}_{zz}%
\end{array}
\right)  \allowbreak\left(
\begin{array}
[c]{c}%
\mathbf{B}_{yu}\\
\mathbf{0}%
\end{array}
\right) \\
&  \left(  \mathbf{R}_{uu}+\beta\mathbf{B}_{yu}^{^{\prime}}\mathbf{P}%
_{yy}\mathbf{B}_{yu}\right)  ^{-1}\beta\left(
\begin{array}
[c]{c}%
\mathbf{B}_{yu}\\
\mathbf{0}%
\end{array}
\right)  ^{T}\left(
\begin{array}
[c]{cc}%
\mathbf{P}_{yy} & \mathbf{P}_{yz}\\
\mathbf{P}_{yz} & \mathbf{P}_{zz}%
\end{array}
\right)  \allowbreak\left(
\begin{array}
[c]{cc}%
\mathbf{A}_{yy} & \mathbf{A}_{yz}\\
\mathbf{0} & \mathbf{A}_{zz}%
\end{array}
\right)
\end{align*}

The matrix to be inverted in the Riccati equation is modified due to
non-controllable variables:%
\[
\left(
\begin{array}
[c]{c}%
\mathbf{B}_{yu}\\
\mathbf{0}%
\end{array}
\right)  ^{T}\left(
\begin{array}
[c]{cc}%
\mathbf{P}_{yy} & \mathbf{P}_{yz}\\
\mathbf{P}_{yz} & \mathbf{P}_{zz}%
\end{array}
\right)  \allowbreak\left(
\begin{array}
[c]{c}%
\mathbf{B}_{yu}\\
\mathbf{0}%
\end{array}
\right)  =\mathbf{B}_{yu}^{T}\mathbf{P}_{yy}\mathbf{B}_{yu}%
\]

This Riccati equation is written as:%

\begin{align*}
&  \left(
\begin{array}
[c]{cc}%
\mathbf{P}_{yy} & \mathbf{P}_{yz}\\
\mathbf{P}_{yz} & \mathbf{P}_{zz}%
\end{array}
\right)  =\left(
\begin{array}
[c]{cc}%
\mathbf{Q}_{yy} & \mathbf{Q}_{yz}\\
\mathbf{Q}_{yz} & \mathbf{Q}_{zz}%
\end{array}
\right) \\
&  +\beta\left(
\begin{array}
[c]{cc}%
\mathbf{A}_{yy}^{T}\mathbf{P}_{yy}\mathbf{A}_{yy} & \mathbf{A}_{yy}^{T}\left(
\mathbf{P}_{yy}\mathbf{A}_{yz}+\mathbf{P}_{yz}\mathbf{A}_{zz}\right) \\
\left(  \mathbf{A}_{yy}^{T}\mathbf{P}_{yy}\mathbf{A}_{yz}+\mathbf{A}_{yy}%
^{T}\mathbf{P}_{yz}\mathbf{A}_{zz}\right)  ^{T} & \mathbf{A}_{yz}^{T}\left(
\mathbf{P}_{yy}\mathbf{A}_{yz}+\mathbf{P}_{yz}\mathbf{A}_{zz}\right)
+\mathbf{A}_{zz}^{T}\left(  \mathbf{P}_{yz}\mathbf{A}_{yz}+\mathbf{P}%
_{zz}\mathbf{A}_{zz}\right)
\end{array}
\right) \\
&  -\allowbreak\beta^{2}\left(
\begin{array}
[c]{c}%
\mathbf{A}_{yy}^{T}\mathbf{P}_{yy}\mathbf{B}_{yu}\\
\mathbf{A}_{yz}^{T}\mathbf{P}_{yy}\mathbf{B}_{yu}+\mathbf{A}_{zz}%
^{T}\mathbf{P}_{yz}\mathbf{B}_{yu}%
\end{array}
\right)  \left(  \mathbf{R}_{uu}+\beta\mathbf{B}_{yu}^{^{\prime}}%
\mathbf{P}_{yy}\mathbf{B}_{yu}\right)  ^{-1}\\
&  \left(
\begin{array}
[c]{cc}%
\mathbf{B}_{yu}^{T}\mathbf{P}_{yy}\mathbf{A}_{yy} & \mathbf{B}_{yu}^{T}\left(
\mathbf{P}_{yy}\mathbf{A}_{yz}+\mathbf{P}_{yz}\mathbf{A}_{zz}\right)
\end{array}
\right)
\end{align*}

where $\mathbf{P}_{yy}$ solves the matrix Riccati equation (Anderson at al. (1996)):%

\[
\mathbf{P}_{yy}\mathbf{=}\mathbf{Q}_{yy}+\beta\mathbf{A}_{yy}^{T}%
\mathbf{P}_{y}\mathbf{A}_{yy}-\beta\mathbf{A}_{yy}^{T}\mathbf{P}_{y}%
\mathbf{B}_{y}\left(  \mathbf{R}_{uu}+\beta\mathbf{B}_{yu}^{^{T}}%
\mathbf{P}_{yy}\mathbf{B}_{yu}\right)  ^{-1}\beta\mathbf{B}_{y}^{T}%
\mathbf{P}_{y}\mathbf{A}_{yy},
\]

where $\mathbf{F}_{y}$ is computed knowing $\mathbf{P}_{y}$:%

\begin{equation}
\mathbf{F}_{y}=-\left(  \mathbf{R}_{uu}+\beta\mathbf{B}_{yu}^{T}%
\mathbf{P}_{yy}\mathbf{B}_{yu}\right)  ^{-1}\beta\mathbf{B}_{y}^{T}%
\mathbf{P}_{y}\mathbf{A}_{yy},
\end{equation}

where $\mathbf{P}_{yz}$ solves the matrix Sylvester equation knowing
$\mathbf{P}_{y}$ and $\mathbf{F}_{y}$ (Anderson et al. (1996)):%

\[
\mathbf{P}_{yz}=\mathbf{Q}_{yz}+\beta\left(  \mathbf{A}_{yy}+\mathbf{B}%
_{y}\mathbf{F}_{y}\right)  ^{T}\mathbf{P}_{y}\mathbf{A}_{yz}+\beta\left(
\mathbf{A}_{yy}+\mathbf{B}_{y}\mathbf{F}_{y}\right)  ^{T}\mathbf{P}%
_{z}\mathbf{A}_{zz}%
\]

\emph{where} $\mathbf{P}_{zz}$, \emph{which is missing in Anderson et al.
(1996), solves the matrix Sylvester equation knowing }$P_{y}$\emph{, }$F_{y}%
$\emph{ and }$P_{yz}$:%

\begin{align*}
\mathbf{P}_{zz}  &  =\mathbf{Q}_{zz}+\mathbf{A}_{yz}^{T}\left(  \mathbf{P}%
_{yy}\mathbf{A}_{yz}+\mathbf{P}_{yz}\mathbf{A}_{zz}\right)  +\mathbf{A}%
_{zz}^{T}\left(  \mathbf{P}_{yz}\mathbf{A}_{yz}+\mathbf{P}_{zz}\mathbf{A}%
_{zz}\right) \\
&  -\allowbreak\beta^{2}\left(  \mathbf{A}_{yz}^{T}\mathbf{P}_{yy}%
\mathbf{B}_{yu}+\mathbf{A}_{zz}^{T}\mathbf{P}_{yz}\mathbf{B}_{yu}\right)
\left(  \mathbf{R}_{uu}+\beta\mathbf{B}_{yu}^{^{\prime}}\mathbf{P}%
_{yy}\mathbf{B}_{yu}\right)  ^{-1}\\
&  \mathbf{B}_{yu}^{T}\left(  \mathbf{P}_{yy}\mathbf{A}_{yz}+\mathbf{P}%
_{yz}\mathbf{A}_{zz}\right)
\end{align*}

\emph{Now, at last, we know }$\mathbf{P}_{zz}$\emph{ so that we can compute
the welfare of Ramsey optimal policy}:

\begin{proposition}
The welfare of Ramsey optimal policy is:%
\[
-\left(
\begin{array}
[c]{c}%
\mathbf{k}_{0}\\
\mathbf{z}_{0}%
\end{array}
\right)  ^{T}\left(
\begin{array}
[c]{cc}%
\mathbf{P}_{kk}-\mathbf{P}_{kk}\mathbf{P}_{xx}^{-1}\mathbf{P}_{xk} &
\mathbf{P}_{kz}-\mathbf{P}_{kx}\mathbf{P}_{xx}^{-1}\mathbf{P}_{xz}\\
\mathbf{P}_{zx}-\mathbf{P}_{zk}\mathbf{P}_{xx}^{-1}\mathbf{P}_{xk} &
\mathbf{P}_{zz}-\mathbf{P}_{zk}\mathbf{P}_{xx}^{-1}\mathbf{P}_{xz}%
\end{array}
\right)  \left(
\begin{array}
[c]{c}%
\mathbf{k}_{0}\\
\mathbf{z}_{0}%
\end{array}
\right)
\]

\end{proposition}

\begin{proof}
Welfare is a function of controllable non-predetermined variables
$\mathbf{x}_{0}$, controllable predetermined variables $\mathbf{k}_{0}$ and
non controllable predetermined auto-regressive shocks $\mathbf{z}_{0}$:%
\[
\left(
\begin{array}
[c]{c}%
\mathbf{x}_{0}\\
\mathbf{k}_{0}\\
\mathbf{z}_{0}%
\end{array}
\right)  ^{T}\left(
\begin{array}
[c]{ccc}%
\mathbf{P}_{xx} & \mathbf{P}_{xk} & \mathbf{P}_{xz}\\
\mathbf{P}_{kx} & \mathbf{P}_{kk} & \mathbf{P}_{kz}\\
\mathbf{P}_{zk} & \mathbf{P}_{zx} & \mathbf{P}_{zz}%
\end{array}
\right)  \left(
\begin{array}
[c]{c}%
\mathbf{x}_{0}\\
\mathbf{k}_{0}\\
\mathbf{z}_{0}%
\end{array}
\right)
\]
Ramsey optimal initial anchor of non-predetermined variables $\mathbf{x}_{0}$
is (Ljungqvist L. and Sargent T.J. (2012), chapter 19):%
\[
\frac{\partial L}{\partial\mathbf{x}_{0}}=\mathbf{P}_{xk}\mathbf{k}%
_{0}+\mathbf{P}_{xx}\mathbf{x}_{0}+\mathbf{P}_{xz}\mathbf{z}_{0}%
=\mathbf{0\Rightarrow x}_{0}=\mathbf{P}_{xx}^{-1}\mathbf{P}_{xk}\mathbf{k}%
_{0}+\mathbf{P}_{xx}^{-1}\mathbf{P}_{xz}\mathbf{z}_{0}%
\]
Hence, the welfare matrix of Ramsey optimal policy is:%
\begin{align*}
&  \left(
\begin{array}
[c]{ccc}%
\mathbf{0} & -\mathbf{P}_{xx}^{-1}\mathbf{P}_{xk} & -\mathbf{P}_{xx}%
^{-1}\mathbf{P}_{xz}\\
\mathbf{0} & \mathbf{I} & \mathbf{0}\\
\mathbf{0} & \mathbf{0} & \mathbf{I}%
\end{array}
\right)  ^{T}\left(
\begin{array}
[c]{ccc}%
\mathbf{P}_{xx} & \mathbf{P}_{xk} & \mathbf{P}_{xz}\\
\mathbf{P}_{kx} & \mathbf{P}_{kk} & \mathbf{P}_{kz}\\
\mathbf{P}_{zk} & \mathbf{P}_{zx} & \mathbf{P}_{zz}%
\end{array}
\right)  \left(
\begin{array}
[c]{ccc}%
\mathbf{0} & -\mathbf{P}_{xx}^{-1}\mathbf{P}_{xk} & -\mathbf{P}_{xx}%
^{-1}\mathbf{P}_{xz}\\
\mathbf{0} & \mathbf{I} & \mathbf{0}\\
\mathbf{0} & \mathbf{0} & \mathbf{I}%
\end{array}
\right) \\
&  =\left(
\begin{array}
[c]{ccc}%
\mathbf{0} & \mathbf{0} & \mathbf{0}\\
-\left(  \mathbf{P}_{xx}^{-1}\mathbf{P}_{xk}\right)  ^{T} & \mathbf{1} &
\mathbf{0}\\
-\left(  \mathbf{P}_{xx}^{-1}\mathbf{P}_{xz}\right)  ^{T} & \mathbf{0} &
\mathbf{1}%
\end{array}
\right)  \left(
\begin{array}
[c]{ccc}%
\mathbf{0} & \mathbf{0} & \mathbf{0}\\
\mathbf{0} & \mathbf{P}_{kk}-\mathbf{P}_{kk}\mathbf{P}_{xx}^{-1}%
\mathbf{P}_{xk} & \mathbf{P}_{kz}-\mathbf{P}_{kx}\mathbf{P}_{xx}%
^{-1}\mathbf{P}_{xz}\\
\mathbf{0} & \mathbf{P}_{zx}-\mathbf{P}_{zk}\mathbf{P}_{xx}^{-1}%
\mathbf{P}_{xk} & \mathbf{P}_{zz}-\mathbf{P}_{zk}\mathbf{P}_{xx}%
^{-1}\mathbf{P}_{xz}%
\end{array}
\right)  \allowbreak\\
&  =\left(
\begin{array}
[c]{ccc}%
\mathbf{0} & \mathbf{0} & \mathbf{0}\\
\mathbf{0} & \mathbf{P}_{kk}-\mathbf{P}_{kk}\mathbf{P}_{xx}^{-1}%
\mathbf{P}_{xk} & \mathbf{P}_{kz}-\mathbf{P}_{kx}\mathbf{P}_{xx}%
^{-1}\mathbf{P}_{xz}\\
\mathbf{0} & \mathbf{P}_{zx}-\mathbf{P}_{zk}\mathbf{P}_{xx}^{-1}%
\mathbf{P}_{xk} & \mathbf{P}_{zz}-\mathbf{P}_{zk}\mathbf{P}_{xx}%
^{-1}\mathbf{P}_{xz}%
\end{array}
\right)  \allowbreak
\end{align*}

\end{proof}

\section{New Keynesian Phillips Curve Example}

The new-Keynesian Phillips curve constitutes the monetary policy transmission
mechanism:%
\[
\pi_{t}=\beta E_{t}\left[  \pi_{t+1}\right]  +\kappa x_{t}+z_{t}\text{ where
}\kappa>0\text{, }0<\beta<1\text{, }%
\]
where $x_{t}$ represents the output gap, i.e. the deviation between (log)
output and its efficient level. $\pi_{t}$ denotes the rate of inflation
between periods $t-1$ and $t$ and plays the role of the vector of
forward-looking variables $\mathbf{x}_{t}$ in the above general case. $\beta$
denotes the discount factor. $E_{t}$ denotes the expectation operator. The
cost push shock $z_{t}$ includes an exogenous auto-regressive component:%
\[
z_{t}=\rho z_{t-1}+\varepsilon_{t}\text{ where }0<\rho<1\text{ and
}\varepsilon_{t}\text{ i.i.d. normal }N\left(  0,\sigma_{\varepsilon}%
^{2}\right)  ,
\]
where $\rho$ denotes the auto-correlation parameter and $\varepsilon_{t}$ is
identically and independently distributed (i.i.d.) following a normal
distribution with constant variance $\sigma_{\varepsilon}^{2}$. The welfare
loss function is such that the policy target is inflation and the policy
instrument is the output gap (Gali (2015), chapter 5):%
\begin{align*}
&  \max-\frac{1}{2}E_{0}%
{\displaystyle\sum\limits_{t=0}^{t=+\infty}}
\beta^{t}\left(  \pi_{t}^{2}+\frac{\kappa}{\varepsilon}x_{t}^{2}\right) \\
\left(
\begin{array}
[c]{c}%
E_{t}\pi_{t+1}\\
z_{t+1}%
\end{array}
\right)   &  =\mathbf{A}\left(
\begin{array}
[c]{c}%
\pi_{t}\\
z_{t}%
\end{array}
\right)  +\mathbf{B}x_{t}+\left(
\begin{array}
[c]{c}%
0_{y}\\
1
\end{array}
\right)  \varepsilon_{t}%
\end{align*}
There is one controllable non-predetermined variable: $\mathbf{x}_{t}=\pi_{t}%
$. There is no controllable predetermined variable ($\mathbf{k}_{t}%
=\mathbf{0}$). Gali's (2015) calibration is:%
\begin{align*}
\sqrt{\beta}\mathbf{A}  &  \mathbf{=}\sqrt{0.99}\left(
\begin{array}
[c]{cc}%
\mathbf{A}_{xx}=-\frac{1}{\beta}=\frac{1}{0.99} & \mathbf{A}_{xz}=-\frac
{1}{\beta}=-\frac{1}{0.99}\\
0 & \mathbf{A}_{zz}=\rho=0.8
\end{array}
\right)  \text{,}\\
\sqrt{\beta}\mathbf{B}  &  \mathbf{=}\sqrt{0.99}\left(
\begin{array}
[c]{c}%
\mathbf{B}_{x}=-\frac{\kappa}{\beta}=-\frac{0.1275}{0.99}\\
\mathbf{B}_{z}=0
\end{array}
\right)  \text{, }\mathbf{Q}\mathbf{=}\left(
\begin{array}
[c]{cc}%
\mathbf{Q}_{xx}=1 & \mathbf{Q}_{xz}=0\\
\mathbf{Q}_{xz}=0 & \mathbf{Q}_{zz}=0
\end{array}
\right)  \text{, }\mathbf{R=}\frac{\kappa}{\varepsilon}=\frac{0.1275}{6}%
\end{align*}
One multiplies matrices by $\sqrt{\beta}$ in order to take the discount factor
in the Riccati equation. The welfare matrix is computed using SCILAB\ lqr
instruction in the appendix:
\[
\mathbf{P=}\left(
\begin{array}
[c]{cc}%
\mathbf{P}_{xx} & \mathbf{P}_{xz}\\
\mathbf{P}_{xz} & \mathbf{P}_{zz}%
\end{array}
\right)  =\left(
\begin{array}
[c]{cc}%
1.7518055 & -1.1389181\\
-1.1389181 & 3.4285107
\end{array}
\right)
\]
Taking into account the optimal initial anchor of inflation ($\pi_{0}=0.65$
for $z_{0}=1$), the welfare matrix is:%
\[
\left(
\begin{array}
[c]{cc}%
0 & -\mathbf{P}_{xx}^{-1}\mathbf{P}_{xz}\\
0 & 1
\end{array}
\right)  ^{T}\left(
\begin{array}
[c]{cc}%
\mathbf{P}_{xx} & \mathbf{P}_{xz}\\
\mathbf{P}_{xz} & \mathbf{P}_{zz}%
\end{array}
\right)  \left(
\begin{array}
[c]{cc}%
0 & -\mathbf{P}_{xx}^{-1}\mathbf{P}_{xz}=0.6504\\
0 & 1
\end{array}
\right)  =\left(
\begin{array}
[c]{cc}%
0 & 0\\
0 & \mathbf{P}_{zz}-\mathbf{P}_{xz}\mathbf{P}_{xx}^{-1}\mathbf{P}_{xz}%
\end{array}
\right)
\]
The welfare loss of Gali's (2015) impulse response functions with Ramsey
optimal initial condition: $\pi_{0}=0.65$ for $z_{0}=1$ is:%
\[
W=-\left(  \mathbf{P}_{zz}-\mathbf{P}_{xz}\mathbf{P}_{xx}^{-1}\mathbf{P}%
_{xz}\right)  z_{0}^{2}=-2.688\cdot z_{0}^{2}%
\]
We found the same value simulating impulse response functions over two hundred
periods, computing period loss function and a discounted sum of these period
loss functions over two hundred periods. Additional results on this example
can be found in Chatelain and Ralf (2019).

\emph{Using only the information available in Anderson et al (1996), e.g.
assuming }$\mathbf{P}_{zz}=0$\emph{ for the missing block matrix in the value
function, welfare loss would be strictly positive }$\mathbf{P}_{xz}%
\mathbf{P}_{xx}^{-1}\mathbf{P}_{xz}=0.74>-2.688$\emph{, which is impossible.}

This paper is part of a broader project which evaluates the bifurcations of
dynamic systems which occurs for Ramsey optimal policy versus discretion
equilibrium (Chatelain and Ralf (2021) or versus simple rules (Chatelain and
Ralf (2020c). In particular, an Hopf bifurcation occurs for the new-Keynesian
model (Chatelain and Ralf (2020a)). Super-inertial interest rate rules are not
solutions of Ramsey optimal monetary policy (Chatelain and Ralf (2020b).
Ramsey optimal policy eliminates multiple equilibria such as the fiscal theory
of the price level in the frictionless model (Chatelain and Ralf (2020d) or in
the new-Keynesian model (Chatelain and Ralf (2020e)).

\section{Appendix}

The numerical solution of the welfare matrix is obtained using Scilab code:

\texttt{beta1=0.99; eps=6; kappa=0.1275; rho=0.8;}

\texttt{Qpi=1; Qz=0 ; Qzpi=0; R=kappa/eps;}

\texttt{A1=[1/beta1 -1/beta1 ; 0 rho] ;}

\texttt{A=sqrt(beta1)*A1;}

\texttt{B1=[-kappa/beta1 ; 0];}

\texttt{B=sqrt(beta1)*B1;}

\texttt{Q=[Qpi Qzpi ;Qzpi Qz ];}

\texttt{Big=sysdiag(Q,R);}

\texttt{[w,wp]=fullrf(Big);}

\texttt{C1=wp(:,1:2);}

\texttt{D12=wp(:,3:\$);}

\texttt{M=syslin('d',A,B,C1,D12);}

[\texttt{Fy,Py]=lqr(M);}

\texttt{Py}

\texttt{Py(2,2)-Py(1,2)*inv(Py(1,1))*Py(1,2)}

\end{document}